\pdfoutput=1
\documentclass[a4paper]{amsart}
 \usepackage{amsmath}
\usepackage{amssymb}
\usepackage{amsthm}
\usepackage{enumerate}
\usepackage[mathscr]{eucal}

\usepackage{amsmath}
\usepackage{amsfonts}
\usepackage{amssymb}
\usepackage[all]{xypic}
\usepackage{hyperref}
\usepackage{amsmath,amsfonts,amsthm,epsfig}
\usepackage{graphicx,epsfig,color}

\input cyracc.def
\font \tencyr=wncyr10
\newfam\cyrfam
\font\tencyr=wncyr10
\font\sevencyr=wncyr7
\font\fivecyr=wncyr5

\textfont\cyrfam=\tencyr \scriptfont\cyrfam=\sevencyr
\scriptscriptfont\cyrfam=\fivecyr

\let\kappa\varkappa

\newcommand{\N}{\mathbb{N}}
\newcommand{\R}{\mathbb{R}}

\newcommand{\C}{\mathcal{C}}
\newcommand{\A}{\mathcal{A}}
\newcommand{\E}{\boldsymbol{E}}

\newcommand{\LL}{\mathcal{L}}

\renewcommand{\O}{\textrm{o}}

\renewcommand{\S}{\boldsymbol{S}}

\newcommand{\x}{\boldsymbol{x}}
\renewcommand{\t}{\boldsymbol{t}}

\newcommand{\REF}[1]{\eqref{#1}}

\newcommand{\df}{\stackrel{\mathrm{def}}{=}}

\newcommand{\virg}[1]{``#1''}

\renewcommand{\kappa}{\mathrm{VSym}}

\theoremstyle{plain}
\newtheorem{theorem}{Theorem}[section]
\newtheorem{prop}[theorem]{Proposition}
\newtheorem{lemma}{Lemma}[section]
\newtheorem{corol}{Corollary}[theorem]

\theoremstyle{definition}
\newtheorem{definition}{Definition}[section]
\newtheorem{remark}{\textnormal{\textbf{Remark}}}

\theoremstyle{remark}

\numberwithin{equation}{section}

\renewcommand{\d}{\mathrm{d}}

\begin{document}
\title[Natural boundary conditions]%
{Natural boundary conditions in geometric calculus of variations}
\author[G. Moreno \and M.E. Stypa]%
{Giovanni Moreno* \and Monika Ewa Stypa**}

\newcommand{\acr}{\newline\indent}

\address{\llap{*\,}Mathematical Institute in Opava\acr
                   Silesian University in Opava\acr
                   Na Rybnicku 626/1\acr
                   746 01 Opava\acr
                   CZECH REPUBLIC}

\email{Giovanni.Moreno@math.slu.cz}

\address{\llap{**\,}Department of Mathematics\acr
                    Salerno University\acr
                    Via Ponte Don Melillo\acr
                    84084 Salerno\acr
                    ITALY}
\email{mstypa@unisa.it}

\thanks{The authors are thankful to the referee for carefully reading the manuscript, and to the organizers of the $17^\mathrm{th}$ summer school in Global Analysis and its Application, held in Levo\v{c}a, August 17--22, for providing a stimulating environment to their research. The first author is also thankful to the   Grant Agency  of the Czech Republic (GA \v CR)
for financial support under  the project P201/12/G028. The second author is grateful to the Ph.D. program  of Salerno University for financial support.}



\begin{abstract}
In this paper we obtain   natural boundary conditions for a large class of variational problems with free boundary values. In comparison with the already existing examples, our framework displays complete freedom   concerning the topology of  $Y$---the       {manifold of dependent and independent variables} underlying a given problem---as well as the order of its Lagrangian. Our result follows from the natural behavior, under boundary--friendly transformations, of an operator, similar to the Euler map,   constructed in the context of   relative horizontal forms on jet bundles (or Grassmann fibrations) over $Y$. 
Explicit examples of natural boundary conditions are obtained when $Y$ is an  $(n+1)$--dimensional domain in $\R^{n+1}$,  and the Lagrangian is first--order (in particular, the hypersurface area).
\end{abstract}

\maketitle
\tableofcontents

\section*{Introduction}
Let $Y$ be a smooth (real) manifold of dimension $n+1$, with nonempty boundary $\partial Y$. 
\begin{definition}\label{defAmmiissibile}
An $n$--dimensional submanifold   $L\subseteq Y$ such that
\begin{enumerate}
\item  $L$ is connected, compact and oriented;
\item $\partial L=L\cap\partial Y$;
\item $ L$ is nowhere tangent to $\partial Y$,
\end{enumerate}
is called \emph{admissible}; the totality of such submanifolds is denoted by     $\A_Y$.
\end{definition}
Introduce a local  coordinate system  $(\x,u)$ on $Y$, where $\x:=(x^1,\ldots,x^n)$. Let  $\lambda=\LL\d^n\x$ be  an $r^\mathrm{th}$ order Lagrangian, i.e., let $\LL=\LL(\x, u,u_I)$ and  $I$  denote a multi--index of length $\leq r$. Suppose that, in such coordinates, an element $L\in\A_Y$ is the graph of a function $u=u(\x)$, defined on  a connected and bounded domain $\Omega\subseteq\R^n$: then    the integral
\begin{equation}\label{eqIntegraleBruttissimo}
\S_\lambda[L]:=\int_\Omega \LL \left(\x,u(\x),\frac{\partial^{|I|} u}{\partial \x^I}\right)\d^n\x
\end{equation}
  makes sense;   it can also be  given a coordinate--free formulation.  %
  
\begin{definition} \label{defProblVarMASTER}
The   \emph{variational problem with free boundary values}  determined by the Lagrangian $\lambda$  on $Y$ consists of finding   the elements of $\A_Y$ which are critical for  $\S_\lambda$.
\end{definition}
Indeed, if properly understood in a geometric framework, $\S_\lambda$ is a real--valued function on $\A_Y$;    the choice of the denomination is justified  by  \REF{eqIntegraleBruttissimo}: if $L$ is allowed to vary within the class $\A_Y$, then the function $u$ describing $L$ is \virg{free} to take any boundary value, as long as $u$ maps $\partial \Omega$ into $\partial Y$.\par
The main theoretical question addressed in this paper is the following: \emph{do the solutions to a variational problem with free boundary values fulfill some extra equation(s) besides the Euler--Lagrange equations}? A positive answer has already been given in \cite{Mor2007,Mor2010,MorenoCauchy}, but without detailed proofs: Section \ref{secRelativeEulerOperator} is devoted to review this result by adding the missing details.\par 
Sections \ref{secFlagFibrations}--\ref{secRelativeEulerOperator} deal with   technical aspects of flag fibrations and relative $\C$--spectral sequences, respectively: the reader not interested in theoretical considerations may skip them, and jump to Corollary \ref{eqTCconforme}, which summarizes their results.    Section \ref{secPreliminare}  explains the key used  to obtain the main result  (Section \ref{secFinale}), namely    the natural behavior of the relative Euler map,   under boundary--friendly transformations. As certainly know  all   who work  in geometric variational calculus and cohomological theory of nonlinear PDEs, the Euler map is but  a small feature of a general theory (comprising, e.g., conservation laws, Helmoltz conditions, hamiltonian structures, recursion operators, etc.), which possesses a natural \emph{relative} analog: we added Sections \ref{secFlagFibrations}--\ref{secRelativeEulerOperator} just to give  a glimpse of it. \par The applicative purpose of this paper is to present explicit  examples of natural boundary conditions.
 In the rather  {pedagogical}  Section \ref{secEsempioMotivante}, we review the classical analytic solution, given by van Brunt in a recent (2006) book  \cite{Brunt}, to one of the  simplest examples of variational problems with free boundary values. More involved examples are suggested by real--life circumstances, 
  as, e.g., the problem of finding the equilibrium of a soap film freely sliding along the inner wall of an arbitrarily--shaped pipe,  discussed in Subsection \ref{CondizioniPerSupericie}\par
The   geometric point of view is the backbone of this paper: besides allowing a transparent formulation of   the main problem, it provides a key tool to obtain a  solution.  Analytic   formulation \REF{eqIntegraleBruttissimo} will be used whenever it is necessary to perform actual computations, as well as a source of valuable insights. For example,  the Euler--Lagrange equations
\begin{equation}
\frac{\delta\LL}{\delta u}=0,
\end{equation}
where $\frac{\delta\LL}{\delta u}$ is the Euler--Lagrange derivative of $\LL$, 
are obtained by a well--known  manipulation of  \REF{eqIntegraleBruttissimo}, 
under the assumption that the variations of $u$   have a compact support in $\overset{\circ}{\Omega}$: hence, a solution to the main problem should be  a stronger condition than the Euler--Lagrange equations themselves. This clue was confirmed by   the discovery of the relative Euler map \cite{Mor2007}, reviewed in Section \ref{secRelativeEulerOperator}.\par
\section{Generalities on geometric calculus of variations}\label{secPreliminare}
The Euler map $\E$ appears, in one form or another,  in all geometric frameworks for Variational Calculus that are based on the language of differential forms on jet spaces (called here \emph{Grassmann fibrations}  following the  recent paper \cite{Zbynek}, of which we also adopt the notation). Building the Grassmann fibration\footnote{In Vinogradov and his school's approach, $G_n^rY$ is denoted by $J^r(Y,n)$, see \cite{MR1857908}} $G^r_n Y$ over $Y$  is just a coordinate--free way to add new coordinates $u_I$, with $|I|\leq r$, to the manifold $Y$, in such a way that $\lambda$ can be considered as an $n$--form on $G^r_n Y$. In this new perspective, \REF{eqIntegraleBruttissimo} can be rewritten without mentioning the local expression of $L$: indeed,  since  $G^r_nL\equiv L$, being $\dim L=n$, the canonical inclusion $\iota_L:L\subseteq Y$ is lifted to an immersion $j^rL:=G^r_n\iota_L:L\longrightarrow G^r_n Y$, which allows to pull any Lagrangian  back to $L$. In other words, \REF{eqIntegraleBruttissimo} reads
\begin{equation}\label{eqIntvarvar}
\S_\lambda:L\in\A_Y\longmapsto \int_L j^rL^\ast\lambda\in\R.
\end{equation}
Passing from \REF{eqIntegraleBruttissimo} to \REF{eqIntvarvar} is far from being a mere aesthetic exercise. It deploys powerful tools to attack  the main problem: essentially, the possibility of using transformations   which mix dependent and independent variables ($\x$ and $u$, respectively,  in the above coordinate system).  In Subsection \ref{CondizioniPerCurve}  we show how a suitable change of coordinates can help  avoid the lengthy computations proposed in Section \ref{secEsempioMotivante}, and how to obtain some useful formulae which, to the authors' opinion, would be very hard (though not impossible) to discover relying on pure analytic methods.\par
The power of transformation methods descends from the natural character of the Euler map: in the principal geometric frameworks for Variational Calculus (Krupka's variational sequences \cite{MR0394755,MR1062026}, Anderson's variational bicomplex \cite{MR590637}, and Vinogradov's $\C$--spectral sequence \cite{MR739952}) the Euler map connects two spaces, say $L(Y)$ and $K(Y)$, containing, respectively,  the Lagrangians and the Euler--Lagrange expressions for the variational problems on $Y$.  We shall not go into the details, since a lot of excellent literature has been written on the subject; nonetheless, we stress that the natural character of the association $Y\longmapsto G_n^rY$, where $r\leq\infty$, i.e., the canonical way to lift transformations of $Y$ to the Grassmann fibrations, makes the associations $Y\longmapsto L(Y), K(Y)$ natural as well.  Indeed, $L(Y)$ and $K(Y)$ are usually defined as quotients of sub--complexes (or sub--sequences) of the de Rham complex of finite (or infinite--order) Grassmann fibrations, and as such they inherit the pull--back from differential forms. In other words,   any diffeomorphism $F:\underline{Y}\longrightarrow Y$, determines a commutative diagram
\begin{equation}\label{eqNaturalita}
\xymatrix{
L(Y)\ar[r]^{\E_Y}   &   K(Y)\\
L(\underline{Y})  \ar[r]^{\E_{\underline{Y}}}  \ar[u]^{F^\ast}&   K(\underline{Y}). \ar[u]^{F^\ast}
}
\end{equation}
If $F$ is a wisely--chosen change of coordinates,  then \virg{the long way} from $L(Y)$ to $K(Y)$, i.e., $F^\ast\circ \E_{\underline{Y}}\circ (F^\ast)^{-1}$ may be more convenient  concerning computations. But this is just a category--theoretic restatement  of the well--known transformation rule for  the Euler--Lagrange equations, which was already known to E. Cartan: the purpose of this paper is to extend it to the class of variational problems with free boundary values, where the Euler--Lagrange equations are sided by the so--called natural boundary conditions, or, equivalently, they are replaced by the \emph{relative Euler--Lagrange equations}.\par
Roughly speaking, the \virg{relative} version of the Euler map arises because  of the boundary $\partial Y$.
\begin{definition}\label{defDefinizioneBordo}
By abuse of notation,\footnote{$\partial G_n^rY$ is more like a prolongation, or lift, to $G_n^rY$, of the boundary $\partial Y$.} we shall put 
 \begin{equation}\label{eqDefinizioneBordo}
\partial G_n^rY:=(\rho^{r,0})^{-1}(\partial Y).\nonumber
\end{equation}
\end{definition}
  Indeed, the canonical inclusion $\iota: \partial G_n^rY \subseteq G_n^rY $ determines a differential algebra epimorphism $\iota^\ast: \Omega( G_n^rY) \longrightarrow \Omega(\partial G_n^rY) $ whose kernel $\ker\iota^\ast:= \Omega( G_n^rY,\partial G_n^rY)$ is, by definition,\footnote{Such a construction is common in Differential Topology (see, e.g., \cite{MR658304}).} the ideal of \emph{relative differential forms} on $G_n^rY$. Much as  $L(Y)$, $K(Y)$, and $\E_Y$ are constructed out of (classes of) differential forms on $G_n^rY$ and natural morphism connecting them, their \virg{relative counterparts}, denoted by  $L(Y,\partial Y)$, $K(Y,\partial Y)$, and $\E^{\mathrm{rel}}_Y$, respectively,  are built out of relative differential forms on $G_n^rY$. Details of this construction, carried out in the context of $\C$--spectral sequences (meaning, in particular, $r=\infty$),  can be found in \cite{Mor2007,Mor2010}\par
Section \ref{secRelativeEulerOperator}   explains why the relative Euler--Lagrange equations 
\begin{equation}\label{eqRelEL}
\E^{\mathrm{rel}}_Y(\lambda)=0
\end{equation}
 represent a solution to the main problem. More precisely, since $K(Y,\partial Y)$ identifies with the direct sum $K(Y)\oplus K(\partial Y)$,  the single equation \REF{eqRelEL} captures two equations simultaneously, viz., the Euler--Lagrange equations
 \begin{equation}\label{eqEL}
\E_Y(\lambda)=0,
\end{equation}
which involve $n$ independent variables, and the natural boundary conditions
 \begin{equation}\label{eqTC}
\E_{Y}^\partial(\lambda)=0,
\end{equation}
where the number of independent variables involved is $n-1$.
Besides providing a common environment for such heterogeneous equations, the formalism of flag fibrations, introduced by the first author in \cite{MorenoCauchy}, and reviewed in Section \ref{secFlagFibrations}, allows to write down \REF{eqTC} in a workable way.\footnote{This paper is based on the talk  \virg{A geometrical framework for Lagrangian theories which involve $n$ and $n-1$ independent variables simultaneously} delivered by the first author on August 24, 2012, within the conference \virg{Variations on a Theme}, dedicated to D. Krupka's seventieth birthday.}
\par
The natural character of relative Euler map follows automatically from its very definition: in other words,   the  \virg{relative} version of diagram \REF{eqNaturalita}, paraphrased by Lemma \ref{lemFormulaTrasformante} below, needs not to be proved. 
\begin{lemma}\label{lemFormulaTrasformante}
Let $F$ be boundary--friendly, i.e., $F(\partial\underline{Y})\subseteq \partial Y$. Then  
\begin{equation}\label{eqFormulaTrasformante}
 \E^{\mathrm{rel}}_{ {Y}}=F^\ast\circ \E^{\mathrm{rel}}_{\underline{Y}}\circ (F^\ast)^{-1}.\nonumber
\end{equation}
\end{lemma}
Lemma \ref{lemFormulaTrasformante}, together with Corollary \ref{corollariosintetizzante}, will be employed  in the last Section \ref{secFinale} to obtain new examples of natural boundary conditions.

\section{A motivating example}\label{secEsempioMotivante}

Let $n=1$,   $Y=[a,b]\times\R$, and $\lambda=\LL(x,u,u')\d x$: in this case,   functions can be identified with their graphs, and  $ \A:=C^\infty([a,b])$ as a subset of $\A_Y$. Hence, up to a (noncritical) restriction of $\A_Y$ to $\A$, the boundary problem with free boundary values determined by $\lambda$ on $Y$, entails finding the functions $u$ such that\footnote{The norm   can be either the  $L^\infty$ or the $H^1$ norm on $C^2([a,b])$.}
\begin{equation}\label{eqPuntoCriticoConNorma}
\underset{\|\hat{u}-u\|\to 0}{\lim}\frac{\S_\lambda[\hat{u}]-\S_\lambda[u]}{\|\hat{u}-u\|}=0.
\end{equation}
In Chapter 7 of van Brunt's   book \cite{Brunt},  the above problem is modified by
allowing  $\hat{u}$ to be defined on a different interval than $u$. To fit this new setting, $\A$ must give up its linear structure and norm, namely  $\A:=\cup_{x_0<x_1}C^\infty([x_0,x_1])$; according, $Y:=\R^2$. 
Despite this,   $\A$ keeps a rather obvious   metric structure $d_\A$. Moreover, with two real numbers $X_0$ and $X_1$  and a suitable function $\xi$, one can construct a  \emph{variation}  $\hat{u}$ of $u$,   whose $d_\A$--distance from $u$ is controlled by a parameter $\epsilon>0$.\par
First, use $X_0$ and $X_1$ to define  a new   interval $[\hat{x}_0,\hat{x}_1]$, where
\begin{equation}
\hat{x}_k=x_k+\epsilon X_k,\quad k=0,1,\nonumber
\end{equation}
and suppose, without loss of generality, that $x_0=\min\{x_0,\hat{x}_0\}$ and\linebreak $\hat{x}_1= \max\{x_1,\hat{x}_1\}$. 
Then, use $\xi\in C^\infty([x_0,\hat{x}_1])$ to construct the  {variation} 
\begin{equation}\label{eqNuovoUColCappello}
\hat{u}:=u^{ {\star}}+\epsilon\xi,
\end{equation}
of $u$, where	$u^{ {\star}}$ is the $2^\textrm{nd}$ order polynomial extension\footnote{To reduce the load of notations, we retain the same symbol  $u$ for the extension    $u^\star$ of $u$. }  of $u$ to the interval $[x_0,\hat{x}_1]$, i.e.,
\begin{align}
	u^{\star}(x)=\left\{\begin{array}{ll} 
u, & x\in[x_0,x_1],\\
	u(x_1)+(x-x_1)u^{\prime}(x_1)+\frac{(x-x_1)^2}{2}u^{\prime\prime}(x_1), & x\in(x_1,\hat{x}_1].
\end{array}\right.\nonumber
\end{align}
Now 
\begin{equation}
d_\A(u,\hat{u}):= ||u-\hat{u}||+|(x_0,u_0)-(\hat{x}_0,\hat{u}_0)|+|(x_1,u_1)-(\hat{x}_1,\hat{u}_1)|\nonumber
\end{equation}
is a well--defined  distance on $\A$, which allows to adapt  \REF{eqPuntoCriticoConNorma}  to the case when the domain of definition of $u$ can be altered:   the norm $\|\hat{u}-u\|$ has to be replaced by the distance $d(u,\hat{u})$. Take the  {variation}  $\hat{u}$    \REF{eqNuovoUColCappello}, and compute
\begin{equation}\label{ExvBruntIneq1}
d_\A(u,\hat{u})\leq ||u-\hat{u}||+ \epsilon X_0 +||u_0-\hat{u}_0||+ \epsilon X_1+||u_1-\hat{u}_1||\leq \epsilon(3\xi + X_0+X_1).
\end{equation}
Inequality \REF{ExvBruntIneq1} shows that
\begin{equation}\label{eqEquazioncinaBellina}
d_\A(u,\hat{u})=\O(\epsilon).
\end{equation}
In order to estimate  the numerator in \REF{eqPuntoCriticoConNorma}, compute\footnote{It is convenient  to write $\LL[u]$ instead of  $\LL(x,u(x),u'(x))$. }
\begin{align}
\S_{\lambda}[\hat{u}]-\S_{\lambda}[u]&=\int_{\hat{x}_0}^{\hat{x}_1}\LL[\hat{u}]\d x-\int_{x_0}^{x_1}\LL[{u}]\d x\label{Difer}\\
&=\int_{x_0+\epsilon X_0}^{x_1+\epsilon X_1}\LL[\hat{u}]\d x-\int_{x_0}^{x_1}\LL[{u}]\d x\nonumber\\
&=\int_{x_0}^{x_1} (\LL[\hat{u}]-\LL[{u}])\d x+ \int_{x_1}^{x_1+\epsilon X_1}\LL[\hat{u}]\d x - \int_{x_0}^{x_0+\epsilon X_0}\LL[\hat{u}]\d x.\nonumber
\end{align}
Equality \REF{Difer} shows that, with respect to the \virg{fixed domain case} \REF{eqPuntoCriticoConNorma},  the variation of $\S_{\lambda} $ in $u$ has two additional contributions due to the variations of the endpoints of the domain of $u$. The main advantage of the geometric approach  presented in Subsection \ref{CondizioniPerCurve} later on, is that such a distinction between the variations of $u$ and the variation of its domain, simply disappear. For the time being, \REF{Difer} can be just rewritten in a more suggestive form 
\begin{equation}\label{eqEquazioneGasante}
\S_{\lambda}[\hat{u}]-\S_{\lambda}[u]= \epsilon \frac{\delta \S_\lambda}{\delta \xi}[u] + \epsilon \frac{\delta \S_\lambda}{\delta X_1}[u]  {-}\epsilon \frac{\delta \S_\lambda}{\delta X_0}[u],\nonumber
\end{equation}
where 
\begin{eqnarray*}
\epsilon \frac{\delta \S_\lambda}{\delta \xi}[u]&=&\int_{x_0}^{x_1} (\LL[\hat{u}]-\LL[u])\d x\\ 
&=&\epsilon\left\{\left.\xi\frac{\partial \LL}{\partial u^{\prime}}[u] \right|^{x_1}_{x_0} + \int_{x_0}^{x_1} \xi\left( \frac{\partial \LL}{\partial u}- \frac{\d}{\d x} \frac{\partial \LL}{\partial u^{\prime}}\right)[u]\d x\right\} + \O(\epsilon^2),
\end{eqnarray*}
and 
\begin{equation}
\epsilon \frac{\delta \S_\lambda}{\delta X_k}[u]= \int_{x_k}^{x_k+\epsilon X_k} \LL[\hat{u}]\d x =  \epsilon X_k \LL(x_k,u(x_k),u^{\prime}(x_k)) + \O(\epsilon^2), \quad k=0,1,\nonumber
\end{equation}
where we used the fact that
$$
\left.\frac{\d}{\d t}\right|_{t=0}\int^{x_k+t}_{x_k}{\LL[u]}\d x=\LL(x,u(x_k),u^{\prime}(x_k))
$$
and 
 $\LL(x,\hat{u}(x_k),\hat{u}^{\prime}(x_k))-\LL(x,u(x_k),u^{\prime}(x_k))=\O(\epsilon^2)$, for $k=0,1$. \par
Now define real numbers $U_0, U_1$   by $\epsilon U_k=\hat{u}(\hat{x}_k)-u(x_k)$, for $k=0,1$, and compute
\begin{eqnarray} 
\epsilon U_k &=& (u+\epsilon\xi)(\hat{x}_k)-u(x_0)= u(x_k) + \epsilon X_k u^{\prime}(x_k) + \epsilon \xi(x_k) + \O(\epsilon^2) -u(x_k)\nonumber\\
&=& \epsilon \left( X_k u^{\prime}(x_k) + \xi(x_k) + \O(\epsilon) \right). \nonumber
\end{eqnarray}
This shows that
\begin{equation}
\xi(x_k)=U_k -X_k u^{\prime}(x_k) + \O(\epsilon),\quad k=0,1.\label{Ex1}
\end{equation}
In view of   \REF{Ex1}, 
   \REF{Difer} reads now
\begin{align*}
\S_{\lambda}[\hat{u}]-\S_{\lambda}[u]{=} &\epsilon \left\{ \left.\xi\frac{\partial \LL}{\partial u^{\prime}}[u] \right|^{x_1}_{x_0} + \int_{x_0}^{x_1} \xi\left( \frac{\partial \LL}{\partial u}- \frac{\d}{\d x} \frac{\partial \LL }{\partial u^{\prime}}\right)[u]\d x\right. \\
&+\sum_{k=0,1}(-1)^k\left[ -X_k\left(\LL-u^{\prime} \frac{\partial \LL}{\partial u^{\prime}}\right)(x_k,u(x_k),u^{\prime}(x_k))\right.\\
&\left. \left.- U_k\frac{\partial \LL}{\partial u^{\prime}}(x_k,u(x_k),u^{\prime}(x_k)) + \xi (x_k)\frac{\partial \LL }{\partial u^{\prime}}(x_k,u(x_k),u^{\prime}(x_k))
\right]\right\}\\ &+ \O(\epsilon^2).
\end{align*}
Equivalently,
\begin{align*}
\S_{\lambda}[\hat{u}]-\S_{\lambda}[u]= &\epsilon \left\{  \int_{x_0}^{x_1} \xi  \frac{\delta \LL }{\delta u}  \d x+\sum_{k=0,1}(-1)^k\right. \\
&\left.\cdot\left[-  X_k\left(\LL-u^{\prime} \frac{\partial \LL }{\partial u^{\prime}}\right)  - U_k\frac{\partial \LL }{\partial u^{\prime}}
\right](x_k,u(x_k),u^{\prime}(x_k))\right\}+ \O(\epsilon^2).
\end{align*}
Plugging the last expression into \REF{eqPuntoCriticoConNorma}, and taking into account \REF{eqEquazioncinaBellina}, we finally see that $u$ is a critical point for 
  $\S_{\lambda}$ if the above term in 
 $\epsilon$ vanishes for all variations of $u$, i.e., for all possible choices of $\xi$, $X_0$ and $X_1$.  In particular,  $u$ must satisfy the ($2^\textrm{nd}$ order) Euler--Lagrange equations, 
\begin{equation}\label{eqEX-EL}
  \frac{\delta \LL }{\delta u}(x,u(x),u^{\prime}(x),u^{\prime\prime}(x))=0, \nonumber
\end{equation}
on its domain of definition, plus a  ($1^\textrm{st}$ order) \emph{natural boundary condition} at the endpoints, 
\begin{equation}\label{eqCondizioneAlBordoNaturalis}
\sum_{k=0,1}{(-1)^{k}} \left[ (U_k-u^{\prime} X_k)\frac{\partial \LL}{\partial u^{\prime}}+X_k\LL\right](x_k,u(x_k),u^{\prime}(x_k))=0.
\end{equation}

Formula \REF{eqCondizioneAlBordoNaturalis} is used in van Brunt's book to prove Theorem  \ref{thTCallaBrunt} below, which answers the main question for one of the simplest  (though nontrivial) examples of a variational problem with free boundary values.
%
%
%
%
\begin{theorem}[Transversality conditions]\label{thTCallaBrunt}
Let $Y\subset\R^2$ be a closed and connected smooth domain, such that $\partial Y$ is the disjoint union of two curves $\gamma_0$ and $\gamma_1$, and $\R^2\smallsetminus Y$ is disconnected, and $\lambda$ be a $1^\mathrm{st}$ order Lagrangian.\footnote{Note that $\A_Y$ is made precisely by all curves lying in $Y$, such that one endpoint belongs to $\gamma_0$ and the other one to $\gamma_1$, without being tangent to any of them.} If an element $u\in\A_Y$ is a solution of the variational problem with free boundary values determined by $\lambda$, then
\begin{enumerate}
\item $u$ obeys the Euler--Lagrange equations on its domain of definition $[x_0,x_1]$;
\item $u$ fulfills the following \emph{transversality conditions}
 \begin{equation}\label{eqAllaBrunt}
\left[\left.\frac{\d y_{\gamma_k}}{\d\sigma}\right|_{\sigma=0} \frac{\partial \LL}{\partial u'}- \left. \frac{\d x_{\gamma_k}}{\d\sigma}\right|_{\sigma=0} \left(u' \frac{\partial \LL}{\partial u'}-\LL  \right) \right](x_k,u(x_x),u'(x_k))=0,\quad k=0,1,
\end{equation}
where $\gamma_k(\sigma)=(x_{\gamma_k}(\sigma),y_{\gamma_k}(\sigma))$ and $\gamma_k(0)=(x_k,u(x_k))$, $k=0,1$.
\end{enumerate}
\end{theorem}
\begin{proof}
 See \cite{Brunt}, Chapter 7.
\end{proof}
   
   In this Section we   observed   the lack of robustness  of   the functional--analytic approach: 
    the slightest change of settings destroyed the norm on the class of admissible functions, and a (in many respects, unnatural) distance appeared in its place, which worked well only after some lengthy    tricks.


 \section{Flag fibrations}\label{secFlagFibrations}

The main motivation for   flag fibrations is that    equations \REF{eqEL} and \REF{eqTC}  involve   $n$ and $n-1$ independent variables, respectively:   merging them into a unique equation requires a new formalism where the number of independent variables can take (at least) two values: $n$ and $n-1$. Recall the fundamental embedding $G^{r}_nY\subseteq G^1_n(G_n^{r-1}Y)$. It allows to regard   a point $\theta\in G^{r}_nY$ as an $n$--dimensional tangent plane\footnote{Called \emph{integral element} by Bryant\&Griffiths \cite{MR1083148}, or   \emph{$R$--plane} by Vinogradov and his school \cite{MR1857908,Sym}.} to $G_n^{r-1}Y$, and an element of   the fibered product
\begin{equation}
P:=G^r_n Y \times_{G^{r-1}_nY} G_{n-1}^1(G_n^{r-1}Y)
\end{equation}
as a pair    consisting of an $n$--dimensional and $(n-1)$--dimensional tangent plane to $G_n^{r-1}Y$ (at the same point). Define
\begin{equation}
F^r_{n,n-1}Y:=\left\{ (\theta_n,\theta_{n-1})\in P\mid \theta_n\supset \theta_{n-1}   \right\}.
\end{equation}
In many respects, the theory of flag fibrations parallels that of Grassmann fibrations; it is useful to review here some of its characteristic  features.

\begin{theorem}
Let $r>0$. Then the following results hold:
\begin{itemize}
\item   $F^r_{n,n-1}Y$ is a smooth manifold, called  the ($r^\textrm{th}$ order) \emph{flag fibration} of $Y$ (of signature $(n,n-1)$): it is   fibered    over the base $Y$, as well as    all lower--order flag fibrations, i.e., the manifolds $F^s_{n,n-1}Y$, with $s<r$.
\item The flag fibration $F^r_{n,n-1}Y$ is naturally fibered over the corresponding (i.e., with the same order $r$ and the same number of independent variables $n$) Grassmann fibration $G_n^rY$.
\item The image  of $F^r_{n,n-1}Y$ under the canonical projection over $G_{n-1}^1(G_n^{r-1}Y)$ is a smooth submanifold, naturally understood as $1^\mathrm{st}$ order nonlinear partial differential equation on $G_n^{r-1}Y$ in $n-1$ independent variables: the   \emph{equation of involutive $(n-1)$--planes} of  $G_n^{r-1}Y$.
\item The equation of involutive $(n-1)$--planes of  $G_n^{r}Y$ projects naturally over $F^r_{n,n-1}Y$.
\item The infinite--order flag fibration $F^\infty_{n,n-1}Y$, obtained as   the inverse limit of finite--order flag fibrations,    identifies with   the equation of involutive $(n-1)$--planes of  $G_n^{\infty}Y$, and, hence, it can be considered as an equation  on $G_n^{\infty}Y$.
\item The   infinite prolongation\footnote{See \cite{KrasVer} for a definition of infinitely prolonged equations.} $(F^\infty_{n,n-1}Y)_{(\infty)}$ of the   $1^\mathrm{st}$ order differential equation $F^\infty_{n,n-1}Y$, understood  as a pro--finite leaf space,\footnote{In the sense of Vinogradov's \virg{Secondary Calculus}: see, for instance, the introduction of   \cite{MR2504466}.}  is naturally interpreted as a space of infinite--order Cauchy data.
\item The infinitely--prolonged equation $(F^\infty_{n,n-1}Y)_{(\infty)}$ fits into a double filtration picture
\begin{equation}
\xymatrix{
&(F^\infty_{n,n-1}Y)_{(\infty)}\ar[dr]^{n} \ar[dl]_{p}& \\
G_n^\infty Y& & G_{n-1}^\infty Y.
}\nonumber
\end{equation}
  mimicking   the similar diagram in the (linear) theory of flag manifolds.
\end{itemize}

\end{theorem}
\begin{proof}
 See \cite{MorenoCauchy}.
\end{proof}

  \begin{theorem}[\cite{MorenoCauchy}, Theorem 9.1]\label{thStrutturale}
 Let $L$ (resp., $\Sigma$) be a leaf (i.e., maximal integral submanifold with respect to the infinite--order contact distribution)  of $G_n^\infty Y$ (resp.,  $G^\infty_{n-1} Y$). Then   the following identifications   \begin{eqnarray}
 p^{-1}(L)& = & G_{n-1}^\infty(L)\nonumber\label{eqThStrutt1ok}\\
 n^{-1}(\Sigma) &=&J^\infty( {N}_{\Sigma})\label{eqThStrutt2ok}
\end{eqnarray}
hold,  where
$ {N}_{\Sigma}$ is a   pro--finite vector bundle called the \emph{infinite--order normal bundle}.\par Moreover, $p$ and $n$ are transverse one to another, in the sense that $ p^{-1}(L)$ (resp., $ n^{-1}(\Sigma)$) maps non degenerately onto $G^\infty_{n-1}Y$  (resp., $G^\infty_nY$).

\end{theorem}

Equality \REF{eqThStrutt2ok} is the less straightforward of the two, and plays a prominent role in the description of the   relative Euler map,  
which will be introduced in the next section.

  \section{Relative Euler operator and natural boundary conditions}\label{secRelativeEulerOperator}

In order to clarify the relationship between relative cohomology and variational problems with free boundary values, recall Definition \ref{defDefinizioneBordo}, and suppose   that $\lambda= {d}\lambda_0$, where $\lambda_0\in\Omega^{n-1}(G^r_nY)$ is such that 
\begin{equation}
\left.\lambda_0\right|_{\partial G_n^rY }=0.\nonumber
\end{equation}
 Then   \REF{eqIntvarvar} reads
 \begin{equation}\label{eqNullLagrang}
\S_\lambda[L]=\int_L j^rL^\ast \lambda=\int_L j^rL^\ast d\lambda_0=\int_{\partial L}\left.j^rL^\ast \lambda_0\right|_{\partial L}=0,
\end{equation}
since $\partial L\subset \partial Y$  according to Definition \ref{defAmmiissibile}. Indeed, $j^rL$ maps $\partial L$ into $\partial G_n^rY $ and, hence, the fact that   $\lambda_0$ vanishes on the latter implies that its pull--back $j^rL^\ast \lambda_0$ vanishes on the former.\par
\begin{lemma}\label{lemmaFondamentale}
 The action $\S_\lambda$ on $\A_Y$ is determined by the equivalence class of $\lambda$ modulo the subspace $d\Omega^{n-1}(G^r_nY, \partial G_n^rY )$ of $\Omega^n(G^r_nY)$.
\end{lemma}
\begin{proof}
 A paraphrase of \REF{eqNullLagrang}.
\end{proof}
%
In order to simplify   further analysis, we shall work, from now on, in the context of infinite Grassmann fibrations and $\C$--spectral sequences; in particular, a Lagrangian  will be a horizontal $n$--form on $G^\infty_nY$, 
\begin{equation}
\lambda\in \Omega_h^n(G^\infty_nY),\nonumber
\end{equation}
where $\Omega_h(G^\infty_nY)$ is the quotient differential algebra of $\Omega(G^\infty_nY)$ with respect to the ideal of contact forms.  An expert in bicomplexes or $\C$--spectral sequences would say that the next  corollary is   the  \virg{horizontalization} of Lemma \ref{lemmaFondamentale} above.
\begin{corol}\label{corFondamentale}
 The action $\S_\lambda$ on $\A_Y$ is determined by the relative horizontal cohomology class   
 \begin{equation}
 [\lambda]_{\textrm{\normalfont rel}}\in  H^n_h(G^\infty_nY, \partial G_n^rY ):=\frac{\Omega_h^{n}(G^\infty_nY)}{d_h\Omega^{n-1}_{h}(G^\infty_nY,\partial G^\infty_nY)},\nonumber
\end{equation}
where $d_h$ is the horizontal differential.
\end{corol}


Corollary \ref{corFondamentale} says precisely that $H^n_h(G^\infty_nY, \partial G_n^rY )$ is the space $L(Y,\partial Y)$ mentioned in Section \ref{secPreliminare}. The   space $K(Y,\partial Y)$ can be obtained  in a similar way,   using   relative forms, contact ideal, and cohomology:  we shall rather   use an approach based on total differential operators and Spencer cohomology,  as  in \cite{Mor2007}.  In the  same cohomological framework it will also appear    the   relative Euler map $\E^\mathrm{rel}_Y$, which allows to obtain the equation \REF{eqRelEL} out of the Lagrangian $\lambda$.\par 
The aim of this section is to prove that \REF{eqRelEL}  is indeed equivalent to the pair of equations \REF{eqEL}--\REF{eqTC} and, furthermore, that either the single equation  \REF{eqRelEL}, or the two coupled equations \REF{eqEL}--\REF{eqTC}, provide a (nontrivial) answer to the main question stated in the Introduction.
 The first result can be found  in \cite{Mor2007}, but its proof, which is a   consequence of Theorem \ref{thStrutturale}, was provided later  in \cite{MorenoCauchy}, and it is a   consequence of the following structural result, which dictates strong restrictions on the topology of the fibration $ \partial G_n^\infty Y\longrightarrow G^\infty_{n-1}(\partial Y)$. 
\begin{corol}[\cite{Mor2007}, Theorem 2]\label{CorThEgfr}
 Consider $G^\infty_{n-1}(\partial Y)$ as a submanifold of $G^\infty_{n-1}(Y)$ via the embedding $\partial Y\subseteq Y$. Then 
 \begin{equation}
 \partial G_n^\infty Y=J_h^\infty( {N}_{G^\infty_{n-1}(\partial Y)})\nonumber
\end{equation}
where $N$ is the infinite--order normal bundle (see  Theorem \ref{thStrutturale}), and $J_h$ means \virg{horizontal jet bundle}.\footnote{Roughly speaking, the analog of jet bundle where derivatives are replaced by total derivatives. See, e.g.,  \cite{MR1857908,MR2504466} for more details.}
\end{corol}
It follows form Corollary \ref{CorThEgfr} that the relative (i.e., constructed with relative forms) $\C$--spectral sequence of $ \partial G_n^\infty Y$ is particularly simple (i.e., one--line\footnote{See \cite{MR1857908} for the meaning of \virg{one--line}}); in turn, this implies that 
 $K(Y,\partial Y)$ splits  into   the sum $K( Y)\oplus K(\partial Y)$ (the proof can be found in \cite{Mor2010}). Hence, equation \REF{eqRelEL}  splits into two equations: \REF{eqEL}  and \REF{eqTC}.\par
The second result has been stated in \cite{MorenoCauchy} without proof, which is provided by Lemma \ref{lemmaTappabuco} below.

\begin{remark}\label{remarkdiconnessione}
 Proposition \ref{propCalcolosaAlgebrica} below contains a general theoretical result concerning relative $\C$--spectral sequences, so that there is no need to   restrict ourselves to the case of one independent variable: in other words,   we let $Y$  to be of dimension $n+m$, where $m$ is arbitrary, i.e., locally,       to be fibered over an $n$--dimensional manifold $X$ with $m$--dimensional fiber (when needed, such a fibration is called $\pi$). Here we recall some terminology.\par
 $\kappa(Y)$ is  the module of vertical symmetries (denoted by $\varkappa$ in \cite{Sym,KrasVer}) of the infinite--order contact distribution on $G_n^\infty Y$, and   $\mathcal{D}$  is  the sub--algebra  of differential operators generated by total derivatives (the $\C$--differential operators, according to \cite{Sym,KrasVer}). Suppose now we work in a local chart (in particular, $\partial Y=\{x_n=0\}$ and  $\pi$ is trivial): in this case,     $D_0^{(j)}$ denotes  the projection on the $j^\textrm{th}$ component of the free $C^\infty( G_n^\infty Y ) $--module $C^\infty( G_n^\infty Y )^m=\kappa(Y)$, and $D_{I}^{(j)}\df D_{I}\circ D_0^{(j)}$ for all $j=1,\ldots,m$, and ${I}$ multi--index of length $n$, i.e., $I\in\N_0^n$. Moreover, $\mathcal{D}(C^\infty( G_n^\infty Y ) ,\Omega_h ^n( G_n^\infty Y ))$ identifies  with $\mathcal{D}(C^\infty( G_n^\infty Y ) ,C^\infty( G_n^\infty Y ) )$ by means of the horizontal volume form $\d^n\x$ , and $\mathcal{D}(C^\infty({\partial G_n^\infty Y }),\Omega_h ^{n-1}({\partial G_n^\infty Y }))$  with $\mathcal{D}(C^\infty({\partial G_n^\infty Y }),C^\infty({\partial G_n^\infty Y }))$ by means of the horizontal volume form $\d_n^{n-1}\x$ on ${\partial G_n^\infty Y }$. 
Accordingly, the formally adjoint modules (see \cite{MR2504466}) $ {\kappa^\dag}(Y)$ and ${\kappa^\dag}({\partial G_n^\infty Y })$ are identified with the dual module of $\kappa(Y)$ and ${\kappa}({\partial G_n^\infty Y })$, which are still free, with bases $\{D_0^{(j)}\}_{j=1,\ldots,m}$ and $\{D_0^{(j,\alpha)}\}_{j=1,\ldots,m, \alpha\in\N_0}$, respectively. \par
Recall that $D_I$ is the composition of total derivatives $D_{x^1}^{i_1}\circ D_{x^2}^{i_2}\circ\cdots\circ D_{x^n}^{i_n}$, with $I=(i_1,\ldots,i_n)$, and, by our own convention, the difference between the multi--index $I$ and an integer $\alpha\leq i_n$ is the multi--index $I-\alpha:=(i_1,\ldots,i_{n-1},i_n-\alpha)$.
\end{remark}

\begin{prop}[On the structure of $K(Y,\partial Y)$]\label{propCalcolosaAlgebrica}
Let $Y$ be as in Remark \ref{remarkdiconnessione}. Then 
 \begin{equation}\label{eqKYDY}
K(Y,\partial Y)=\frac{\mathcal{D}(\kappa(Y),\Omega_h ^n(G_n^\infty Y))}{\delta\left(\mathcal{D}(\kappa(Y),C^\infty( G_n^\infty Y ) )\otimes\Omega_h ^{n-1}(G_n^\infty Y ,{\partial G_n^\infty Y })\right)},
\end{equation}
where $\delta$ is the Spencer differential. Moreover, the cohomology class of the cocycle $\square=a^{I}_jD_{I}^{(j)} \in \mathcal{D}(\kappa(Y),\Omega_h ^n(G_n^\infty Y))$ 
 is identified with the pair $(\E(\square),\E^\partial(\square))$ $ \in {\kappa^\dag(Y)}\oplus{\kappa^\dag}({\partial G_n^\infty Y })$, where $ \E(\square)= (-1)^{|{I}|}D_{I}(a_j^{I})D_0^{(j)},$
%
%
%
%
 and
 \begin{equation}
 \E^\partial(\square)=\sum_{{I}\in\N_0^n}(-1)^{|{I}|-{i}_n}\sum_{j=1}^m\sum_{\alpha<i_n} (-1)^\alpha D_{{I}-i_n}(D_{x_n}^\alpha(a^{I})_{|{\partial G_n^\infty Y }}) D_0^{(j,{i}_n-\alpha-1)}    \label{eqBellaIdentificazione}
\end{equation}
\end{prop}
\begin{proof}
 By the definition of relative $\C$--spectral sequences  \cite{Mor2007},   the space $K(Y,\partial Y)$  is the   $n^\textrm{th}$  cohomology space of the subcomplex 
 \begin{equation}
\mathcal{D}(\kappa(Y),C^\infty( G_n^\infty Y ) )\otimes\Omega_h (G_n^\infty Y,{\partial G_n^\infty Y })\nonumber
\end{equation}
of $\mathcal{D}(\kappa(Y),\Omega_h(G_n^\infty Y) )$. Expression \REF{eqKYDY} is a consequence of the fact that   $\Omega_h ^n(G_n^\infty Y ,{\partial G_n^\infty Y })$ equals $\Omega_h ^n(G_n^\infty Y)$. 
In other words, $K(Y,\partial Y)$ has the same $n$--cocycles as $K(Y )$, but fewer $n$--coboundaries, which explains why the  $n^\textrm{th}$  cohomology of the subcomplex turns out to be quite larger than the cohomology of the entire complex: in turn, this explains the appearance of natural boundary conditions. For the sake of simplicity, we shall skip the index $j$.\par
We now prove   that the relative Spencer cohomology of $\square$ is identified with \REF{eqBellaIdentificazione}.
To this end, observe that the elements $\d_i^{n-1}\x$, for $i=1,\ldots,n-1$, together with $x_n\d_n^{n-1}\x $, form a basis for  $\Omega_h ^{n-1}(C^\infty( G_n^\infty Y ) ,{\partial G_n^\infty Y })$. Accordingly, elements of $\mathcal{D}(C^\infty( G_n^\infty Y ) ,\Omega_h ^{n-1}(C^\infty( G_n^\infty Y ) ,{\partial G_n^\infty Y }))$ can be obtained by summing up the elements  $\square^i\otimes\d_i^{n-1}\x $ and $\square^n\otimes x_n\d_n^{n-1}\x $, with $\square^i\in\mathcal{D}(C^\infty( G_n^\infty Y ) ,C^\infty( G_n^\infty Y ) )$ for $i=1,\ldots,n$.\par
We compute the Spencer differential $\delta(\square^i\otimes\d_i^{n-1}\x )=(-1)^{i-1}(D_{1_i}\circ\square^i)\otimes\d^n\x  $, and observe that  \begin{align*}
(\delta(\square^n\otimes x_n\d_n^{n-1}\x ))(f)&=  {d_h}(\square^n(f)x_n\d_n^{n-1}\x )\\
&= {d_h}(\square^n(f))\wedge(x_n\d_n^{n-1}\x )+\square^n(f) {d}(x_n\d_n^{n-1}\x )\\
&=D_{1_n}( \square^n(f))dx_n\wedge x_n\d_n^{n-1}\x +\square^n(f)dx_n\wedge\d_n^{n-1}\x \\
&=(-1)^{n-1}(x_nD_{1_n}\circ\square^n+\square^n)(f)\d^n\x,  
\end{align*}
for arbitrary $f\in C^\infty( G_n^\infty Y ) $, whence
\begin{align*}
\delta(\square^n\otimes x_n\d_n^{n-1}\x ) &= (-1)^{n-1}((x_nD_{1_n}+1)\circ\square^n)\otimes\d^n\x  \\
&= (-1)^{n-1}(D_{1_n}\circ x_n\circ \square^n)\otimes\d^n\x,  
\end{align*}
since $1=[D_{1_n},x_n]$.\par
It follows that $a\circ D_{1_i}\circ\square$ is cohomologous to $- D_{1_i}(a)\circ\square$ for all $i=1,\ldots,n-1$, whereas $a\circ D_{1_n}\circ\square=(- D_{1_n}(a)+D_{1_n}\circ a)\circ\square$ is not generally cohomologous to $- D_{1_n}(a)\circ\square$, since $D_{1_n}\circ a\circ\square$ is not a coboundary, unless $a$ factors through $x_n$.\par
Take now $\square=a^{I} D_{I}\in\mathcal{D}(C^\infty( G_n^\infty Y ) ,\Omega_h ^n)$, 
with ${I}\in\N_0^n$. Such an operator $\square$  is cohomologous to the operator
\begin{align*}
\square'&=(-1)^{|{I}|-{i}_n}D_{{I}-{i}_n }(a^{I})\circ D_{x_n}^{{i}_n}\\
&= (-1)^{|{I}|-{i}_n} (-D_{x_n}(D_{{I}-{i}_n }(a^{I}))+D_{x_n}\circ D_{{I}-{i}_n }(a^{I}))\circ D_{x_n}^{{i}_n-1}\\
&=(-1)^{|{I}|-({i}_n-1)} D_{{I}-({i}_n-1) }(a^{I})\circ D_{x_n}^{{i}_n-1}+(-1)^{|{I}|-{i}_n} D_{x_n}\circ D_{{I}-{i}_n }(a^{I})\circ D_{x_n}^{{i}_n-1}\\
&\ \vdots\quad{i}_n\ \textrm{iterations}\nonumber\\
&=(-1)^{|{I}|} D_{{I}}(a^{I})+(-1)^{|{I}|-{i}_n} D_{x_n}\circ D_{{I}-{i}_n }(a^{I})\circ D_{x_n}^{{i}_n-1}\\
&+(-1)^{|{I}|-({i}_n-1)} D_{x_n}\circ D_{{I}-({i}_n-1) }(a^{I})\circ D_{x_n}^{{i}_n-2}+\ldots,
\end{align*}
i.e., to
\begin{equation}
\square'=(-1)^{|{I}|} D_{{I}}(a^{I})+\sum_{0<\alpha\leq i_n} (-1)^{|{I}|-\alpha} D_{x_n}\circ D_{{I}-\alpha}(a^{I})\circ D_{x_n}^{\alpha-1},\nonumber
\end{equation}
which turns into   a coboundary if and only if the function $D_{{I}}(a^{I})$ is zero and all the functions $ D_{{I}-\alpha}(a^{I})$ factor through $x_n$, i.e., they vanish on ${\partial G_n^\infty Y }$.\par
This means that the cohomology class $[\square]$ is uniquely determined by\linebreak $(-1)^{|{I}|}D_{{I}}(a^{I})\in C^\infty( G_n^\infty Y ) $, i.e., by $\E(\square)$ and by the set of functions\linebreak $ (-1)^{|{I}|-\alpha}D_{{I}-\alpha}(a^{I})_{|{\partial G_n^\infty Y }}\in C^\infty({\partial G_n^\infty Y })$, with $\alpha=1,\ldots,{i}_n$. Notice that the latter ones can be rewritten as $D_{{I}-i_n}(D_{x_n}^{\alpha}(a^{I})_{|{\partial G_n^\infty Y }})$, with $\alpha=0,\ldots,{i}_n-1$, because the multi--index ${I}-{i}_n $ belongs actually to $\N_0^{n-1}$ and the first $n-1$ total derivatives are tangent to ${\partial G_n^\infty Y }$: hence the last set of functions represent the coordinates of the element $\E^\partial(\square)$, according to \REF{eqBellaIdentificazione}.
\end{proof}

\begin{prop}[On natural boundary conditions]\label{propCalcolosaAnalitica}
 Let $L$ be locally given by the graph of a function $u=u(\x)$, defined on a compact and connected subset $\Omega\subset\R^n$, such that $\partial\Omega$ has equation $x^n=0$. If $L$ is critical for $\S_\lambda$, then  the following equations hold on $\partial\Omega$:
 \begin{equation}
\sum_{|I|\leq r,\ i_n>\alpha} (-1)^{|I|-\alpha-1}\frac{\d^{|I|-i_{n}}}{\d x ^{I-i_{n}}}\left( \left. \frac{\d^{i_n-\alpha-1} }{\d x_n^{i_n-\alpha-1}}\left(\frac{\partial \LL}{\partial u_I}\right)\right|_{\partial\Omega} \right)=0,\quad \alpha=0,\ldots,r-1.\nonumber
\end{equation}
 
\end{prop}
\begin{proof}
 Let $r$ be the order of $\lambda$ and $I=(i_1,\ldots,i_n)$. Put
\begin{equation}
\LL^{i_1i_2\cdots i_n}:=\frac{\partial\LL}{\partial u_I}=\frac{\partial \LL}{\partial u_{x_1^{i_1}x_2^{i_2}\dots x_n^{i_n}}}.\nonumber
\end{equation}
From the well--known formula of elementary calculus
\begin{equation}\label{eqMagica}
fg^{(i)}=\left(  \sum_{s=0}^{i-1}f^{(s)}g^{(i-s-1)}\right)^\prime+(-1)^if^{(i)}g
\end{equation}
it follows that
\begin{eqnarray}
&&\int_\Omega{\sum_{i_1+\cdots+i_n\leq r}{\LL^{i_1i_2\cdots i_n}\eta_{x_1^{i_1}x_2^{i_2}\dots x_n^{i_n}}}\d^n\x}\label{eqIncrebibile_0}\\
&&=\int_\Omega{
\sum_{i_1+\cdots+i_n \leq r}{\left\{
 {\frac{\d}{\d x_1}\left(\sum_{s_1=0}^{i_1-1}{(-1)^{s_1}
\frac{\d^{s_1} \LL^{i_1i_2\cdots i_n}}{\d x_1^{s_1}}  \eta_{x_1^{i_1-s_1-1}x_2^{i_2}\dots x_n^{i_n}} }\right) } 
 \right.} 
}\label{eqIncrebibile_1}\\
&& \left.+ {
 {\left(
  (-1)^{i_1}\frac{\d^{i_1} \LL^{i_1i_2\cdots i_n}}{\d x_1^{i_1}} \eta_{x_2^{i_2}x_3^{i_3}\dots x_n^{i_n}
}
\right)}} \right\}\d^n\x\label{eqIncrebibile_2}
\end{eqnarray} 
Then, applying again \REF{eqMagica} to the term \REF{eqIncrebibile_2}, we obtain
\begin{eqnarray}
&\REF{eqIncrebibile_0} &=\REF{eqIncrebibile_1}\nonumber\\
&&+\int_\Omega 
\sum_{i_1+\cdots+i_n \leq r} \left\{
     \frac{\d}{\d x_2}\left(\sum_{s_2=0}^{i_2-1} (-1)^{s_2}
\frac{\d^{s_2}}{\d x_2^{s_2}}\right.\right.\nonumber\\
&&\cdot\left.\left.\left((-1)^{i_1}\frac{\d^{i_1} \LL^{i_1i_2\cdots i_n}}{\d x_1^{i_1}} \right) \eta_{x_2^{i_2-s_2-1}\dots x_n^{i_n}}  \right)   \right.\label{eqIncrebibile_3}\\
&&  \left.+ 
  (-1)^{i_2}\frac{\d^{i_2}}{\d x_2^{i_2}}\left((-1)^{i_1}\frac{\d^{i_1}\LL^{i_1i_2\cdots i_n}}{\d x_1^{i_1}} \right) \eta_{x_3^{i_3}x_4^{i_4}\dots x_n^{i_n}}
\right\} \d^n\x\label{eqIncrebibile_4}
\end{eqnarray} 
Again, by  \REF{eqMagica}, we develop term  \REF{eqIncrebibile_4}:
\begin{eqnarray} 
&\REF{eqIncrebibile_0} &=\REF{eqIncrebibile_1}+\REF{eqIncrebibile_3}+\cdots\label{eqIncrebibile_5}\\
&&+\int_\Omega \sum_{i_1+\cdots+i_n \leq r} \left\{  \frac{\d}{\d x_n}\left(\sum_{s_n=0}^{i_n-1} (-1)^{s_n}
\frac{\d^{s_n}}{\d x_n^{s_n}}\left((-1)^{i_{n-1}}\right.\right.\right.\nonumber\\
&&\cdot\left.\left.\left.\frac{\d^{i_{n-1}}}{\d x_{n-1}^{i_{n-1}}} \left(\dots \left((-1)^{i_1}\frac{\d^{i_1}\LL^{i_1i_2\cdots i_n}}{\d x_1^{i_1}}  \right) \dots \right)\right) \eta_{x_n^{i_n-s_n-1}} \right) \right.\label{eqIncrebibile_6}\\
&&+{{\left.  {(-1)^{i_n}\frac{\d^{i_n}}{\d x_n^{i_n}}\left(\dots \left((-1)^{i_1}\frac{\d^{i_1} \LL^{i_1i_2\cdots i_n}}{\d x_1^{i_1}}  \right)\dots \right) \eta } \right\}}\d^n\x.
}\label{eqIncrebibile_7}
\end{eqnarray} 
Since $\partial\Omega =\{x_n=0\}$, all terms appearing on line \REF{eqIncrebibile_5} disappear, being of the form
\begin{equation*}
\int_\Omega  \frac{\partial}{\partial x_i}(F)\d^n\x=\int_{\Omega}  d(F\d_i^{n-1}\x)=\int_{\partial\Omega} \left. F\d_i^{n-1}\x\right|_{\partial\Omega}=0,\quad \forall i=1,2,\ldots,n-1.
\end{equation*}
On the other hand, \REF{eqIncrebibile_7} is the Euler--Lagrange; it remains just \REF{eqIncrebibile_6}, i.e.,
 \begin{equation}\label{eqQUASIFinale}
\int_{\partial\Omega} \left.
\sum_{|I|\leq r}\sum_{s_n=0}^{i_n-1}(-1)^{|I|-i_n+s_n}\frac{\d^{|I|-i_{n}}}{\d x ^{I-i_{n}}}\left(  \frac{\d^{s_n}  \LL^{i_1i_2\cdots i_n}  }{\d x_n^{s_n}}  \right)\eta_{x_n^{i_n-s_n-1}}\right|_{\partial\Omega} \d_n^{n-1}\x.
\end{equation}
Since \REF{eqQUASIFinale}  has to vanish for all variations $\eta$,  all equations \REF{eqVERAMENTEFinale} must be satisfied.
%
%
\end{proof}

\begin{lemma}\label{lemmaTappabuco}
 Let $L\in\A_Y$ be a solution to the variational problem with free boundary values determined by $\lambda$ on $Y$. Then equation \REF{eqEL} holds on $L$ and equation   \REF{eqTC}   holds on $\partial L$.
\end{lemma}

Before providing a proof, it is convenient to cast a bridge between the approach based on   {total differential operators} to the space $K(Y)$, sketched in Remark  \ref{remarkdiconnessione}, and a perhaps more familiar one, based on \virg{$1$--contact, $n$--horizontal} $(n+1)$--forms, or forms \virg{of type $(1,n)$}. Namely,   \REF{eqEL} can be written down as
\begin{equation}\label{eqELconforme}
\E_Y(\lambda)=\frac{\delta \LL}{\delta u}\omega\wedge\d^n\x,
\end{equation}
where $\omega$ is the zero--order contact form, and  $\omega\wedge\d^n\x$     plays the role of the generator $D_0$  (see in Remark \ref{remarkdiconnessione}) of the module $K(Y)$. Equation \REF{eqELconforme} clarifies  the above   sentence \virg{\REF{eqEL}  holds on $L$}:  it means  that \REF{eqELconforme}, {pulled back to $L$ via $j^\infty L$}, vanishes. \par
Similarly, the results contained in Corollary \ref{CorThEgfr} and Proposition \ref{propCalcolosaAlgebrica} give a solid basis to the sentence  \virg{holds on $\partial L$}, since
\begin{equation}\label{eqTCconforme}
\E^\partial_Y(\lambda)=\E^\partial_{Y,\alpha}(\lambda)\omega^\alpha\wedge\d_{n-1}^n\x,
\end{equation}
where now the $\omega^\alpha$'s are the zero--order contact forms on $J_h^\infty( {N}_{G^\infty_{n-1}(\partial Y)})$, and $\omega^\alpha\wedge\d_{n-1}^n\x$ plays the role of the generators $D_0^{j,\alpha}$, where $j=1$  (see in Remark \ref{remarkdiconnessione}), of the module $K(\partial G_n^\infty Y)$. \par
According,  $\E^\partial_{Y,\alpha}(\lambda)$ can be obtained as the coefficient of  $D_0^{1,\alpha}$, in   \REF{eqBellaIdentificazione}, where $j =1$, and $a^I=\frac{\partial \LL}{\partial u_I}$.
%
%
%
 Again, the sentence \virg{\REF{eqTC}   holds on $\partial L$} means that \REF{eqTCconforme}, pulled back to $\partial L$ via $j_\infty\partial L$, vanishes (see also Theorem 11.1  in \cite{MorenoCauchy}).

\begin{proof}

 The first fact is obvious: if $L$ a solution of a variational problem with free boundary values, then $L\smallsetminus\partial L$ is a solution of the Euler--Lagrange equation determined by the same Lagrangian $\lambda$, i.e., equation \REF{eqEL} holds.\par
 We stress that, in order to prove the second fact, it is necessary to have the result on the structure of equation \REF{eqTC}   provided by Corollary \ref{CorThEgfr}. Namely,  \REF{eqTC} is localizable, in the sense that its left--hand side belong to a  module of sections, and hence it vanishes locally if and only if it vanishes globally. Then, we can  
 choose a coordinate system $(\x,u)$ such that $L$ is the graph of a function $u=u(\x)$ on $\Omega$ and $\partial Y=\{x_n=0\}$. 
 Since $L$ is critical, all equations \REF{eqVERAMENTEFinale} must hold true on $\partial \Omega$; on the other hand, the above discussions showed  that   equations  \REF{eqVERAMENTEFinale} are nothing but $\E^\partial_{Y,\alpha}=0$: hence,    \REF{eqTCconforme} vanishes, i.e.,  \REF{eqTC}  must be valid on $\partial L$. 
 
\end{proof}

For  readers   not interested in theoretical details, we collect the main result of the last two sections into a convenient (though redundant) Corollary.

\begin{corol}[A solution of the main problem]\label{corollariosintetizzante}
  Let $L\in\A_Y$ be a solution to the variational problem with free boundary values determined by $\lambda$ on $Y$. Then,   the natural boundary conditions
 $\E_Y^\partial(\lambda)=0$ are satisfied on $\partial L$. In local coordinates, $\E^\partial_Y(\lambda)=\E^\partial_{Y,\alpha}(\lambda)\omega^\alpha\wedge\d_{n-1}^n\x$, where:
 \begin{itemize}
\item[i)]  the $\omega^\alpha$'s are the zero--order contact forms on the infinite jet  of a suitable pro--finite vector bundle over $\partial Y$ which arises in the theory of flag fibrations over $Y$;
\item[ii)] if $\alpha$ is less than the order of the Lagrangian $\lambda$, and $L$ is locally the graph of a function $u=u(\x)$ on $\Omega\subseteq\R^n$, the component $\E^\partial_{Y,\alpha}(\lambda)$ is given by
  \begin{equation}\label{eqVERAMENTEFinale}
\E^\partial_{Y,\alpha}(\lambda)=\sum_{|I|\leq r,\ i_n>\alpha} (-1)^{|I|-\alpha-1}\frac{\d^{|I|-i_{n}}}{\d x ^{I-i_{n}}}\left( \left. \frac{\d^{i_n-\alpha-1} }{\d x_n^{i_n-\alpha-1}}\left(\frac{\partial \LL}{\partial u^I}\right)\right|_{\partial\Omega} \right).
\end{equation}
\end{itemize} 
\end{corol}

 \section{Applications}\label{secFinale}

Together, Lemma \ref{lemFormulaTrasformante} and Corollary \ref{corollariosintetizzante} provide a powerful tool for writing down concrete examples of  natural boundary conditions.
 Computations presented in this section will be simplified by   some \virg{tricks} based on multi--linear algebra and   total differentials (Remarks \ref{remFinale1} and \ref{remFinale2} below). 
\begin{remark}[Top differential forms]\label{remFinale1}
A brute--force attempt to change  variables in  a  multi--dimensional integral may lead to a meaningless formula
\begin{equation}\label{eqCheNonHaSenso}
\int f(\x)\d^n\x = \int f(\x(\t))\frac{\d^n\x}{\d^n\t}\d^n\t.
\end{equation}
Nonetheless, since the $C^\infty(\R^n)$--module $\Omega^n(\R^n)$  is freely generated by $\d^n\t$, any $n$--form can be identified with a function. In particular, $\d^n\t$ identifies with 1, and $\d^n\x$ with the Jacobian of the change of variables $\x=\x(\t)$, thus recovering the meaning of \REF{eqCheNonHaSenso}. From now on,  all $n$--forms will be identified with functions: hence, an expression like $\Xi(\omega)$, where $\Xi$ is a vector field and $\omega$ an $n$--form, is \emph{not} the Lie derivative of $\omega$, but the function $\Xi(f)$, where $f$ is uniquely defined by $\omega=f\d^n\t$. 
\end{remark}

\begin{remark}[Total differentials]\label{remFinale2}
 Formula \REF{eqCheNonHaSenso} can be adapted to variational integrals, just by replacing differentials by total differentials, namely
 \begin{align}
\int f(\x, u(\x), u_1(\x),\ldots, u_n(\x))\overline{\d}^n\x =\nonumber\\
 \int f(\x(\t, y),u(\t,y), u_1(\t, y, y_1,\ldots,y_n),\ldots, u_n(\t, y, y_1,\ldots,y_n))\frac{\overline{\d}^n\x}{\overline{\d}^n\t}\overline{\d}^n\t\nonumber\label{eqCheNonHaSensoOrizz}
\end{align}
 where now $\overline{\d}^n\x=\overline{d}x^1\wedge \cdots\wedge \overline{d}x^n$ (i.e., the operator $d_h$ used in Section \ref{secRelativeEulerOperator}). Recall that
 \begin{equation}
\overline{d}x=D_{t^i}(x)dt^i, \quad x=x(\t)\nonumber
\end{equation}
where $D_{t^i}=\partial_{t^i}+y_i\partial_u$  is the total derivative operator with respect to $t^i$. In this context, horizontal $n$--forms on $G_n^1Y$, i.e., the space with coordinates $(\t,y,y_1\ldots,y_n)$, are identified with functions on the same space, via the horizontal volume form $\overline{\d}^n\t$. Accordingly, $\overline{\d}^n\x$ is the \virg{total Jacobian} associated with the change of variables $(\x,u)=(\x(\t,y),u(\t,y))$.

\end{remark}

 \subsection{A $1^\textrm{st}$ order,  one--dimensional example}\label{CondizioniPerCurve}

Consider again the variational problem with free boundary values of Theorem \ref{thTCallaBrunt}, Section \ref{secEsempioMotivante}. Let  $\LL=\LL(t,y,y')$ be its Lagrangian, and recall that $\partial Y$ is the disjoint union of two curves in the $(t,y)$--plane. Then, if  $\gamma(\sigma)=(t_\gamma(\sigma), y_\gamma(\sigma))$ is one of them, a critical point $y=y(x)$ for $\S_\lambda$ must fulfill the natural boundary condition
%
%
%
%
%
%
%
\begin{equation}\label{eqEsercFin}
\left[y'_{\gamma}(0) \frac{\partial \LL}{\partial y'}-t'_{\gamma}(0) \left(y'\frac{\partial \LL}{\partial y'}-\LL\right)\right](t, y(t), y'(t))=0
\end{equation}
where $(t,y(t))=\gamma(0)$ (see \REF{eqAllaBrunt}).\par
%
Equation \REF{eqEsercFin} can be obtained in a transparent geometrical way, without introducing \emph{ad hoc} metrics on the set $\A$. Just use  a change of coordinates $(t,y)\stackrel{F}{\longmapsto} (x,u)$,
\begin{eqnarray*}
x&=&x(t,y)\\
u&=&u(t,y)
\end{eqnarray*}
 which \virg{rectifies} the curve $\gamma$, i.e., such that $F_\ast(\gamma)$ is, for instance, the $u$--axis of the $(x,u)$--plane. Then,  lift $F$   to  a contact transformation $(t,y,y')\stackrel{F}{\longmapsto} (x,u,u')$ of $G_1^1Y$, where 
\begin{equation}
u'=\frac{u_t+y'u_y}{x_t+y'x_y}.\label{eqSollevamento}
\end{equation}
It is a simple exercise to get \REF{eqSollevamento} (see, for instance, \cite{Sym}, Section 1.2); nevertheless, in view of the next generalization, we prefer to justify it, by using the total differential operator $\overline{d}$. Recall that 
\begin{equation}\label{eqCretina}
 \overline{d}f=D_t(f)\d t,
\end{equation}
 for $f=f(t,y,y')$, with $D_t=\partial_t+y'\partial_y$ being the total derivative operator in $t$. As announced in Remark \ref{remFinale2}, we shall identify horizontal one--forms on the $(t,y,y')$--space with functions: hence, \REF{eqSollevamento} above can be written as
\begin{equation}
u'=\frac{\overline{d}u}{\overline{d}x}=\frac{D_t(u)\d t}{D_t(x)\d t}.\label{eqSollevamentoPiuBella}
\end{equation}
It is worth noticing that the inverse transformation $F^{-1}$ is  the same as \REF{eqSollevamentoPiuBella}
\begin{equation}
y^\prime=\frac{\overline{d}y}{\overline{d}t}=\frac{D_x(y)\d x}{D_x(t)\d x}=\frac{y_x+u'y_u}{t_x+u't_u},\label{eqFormulettaChecivoleva}
\end{equation}
where now  the total derivative operator is taken with respect to $x$. 
%
%
%
It follows that 
\begin{equation}\label{eqExercisio2}
\frac{\partial y'}{\partial u'}=\frac{\partial }{\partial u'}\left(\frac{D_x(y) }{D_x(t) }\right).
\end{equation}
Finally, since  $\d t=\overline{d}t$ (see   \REF{eqCretina}),   the Lagrangian $\lambda$       reads
\begin{equation}
\lambda=\LL(t,y,y')\overline{d}t=\LL(t,y,y') D_x(t)\d x\nonumber
\end{equation}
in the $(x,u,u')$--space, 
where $D_x(t)$ plays the role of \virg{total Jacobian}  (Remark \ref{remFinale2}). In other words,   $(F^\ast)^{-1}(\lambda) =\widetilde{\LL}\d x$, where $\widetilde{\LL}=\widetilde{\LL}(x,u,u')$ is given by
%
%
%
\begin{equation}
\widetilde{\LL}=(F^\ast)^{-1}(\LL)D_x(t) =\LL(t(x,u),y(x,y),y'(x,u,u')) D_x(t)(x,u,u').\label{eqLtilde}
\end{equation}
Now everything is ready.   $(F^\ast)^{-1}(\lambda)$ determines a variational problem with free boundary values on $F(Y)$ and $\partial F(Y)=F(\partial Y)$, by construction, consists of two curves, one of which is the $u$--axis:   by Corollary \ref{corollariosintetizzante},  on such an axis  the natural boundary conditions  take a particularly simple form
\begin{equation}\label{eqEsempio1cambiocoordinate}
 \frac{\partial \widetilde{\LL}}{\partial u'} (0,u(0),u'(0))=0.\nonumber
\end{equation}
It remains to express \REF{eqEsempio1cambiocoordinate} in terms of the coordinates $(t,y,y')$, i.e.,   to apply Lemma  \ref{lemFormulaTrasformante}.
%
%
%
%
%
%
We compute
\begin{align}
 \frac{\partial \widetilde{\LL}}{\partial u'}&\stackrel{\REF{eqLtilde}}{=} \frac{\partial }{\partial u'}(\LL D_x(t))\nonumber\\
 &=\frac{\partial \LL}{\partial u'}  D_x(t)+\LL\frac{\partial D_x(t)}{\partial u'}\nonumber\\
 &\stackrel{\REF{eqExercisio2}}{=} \frac{\partial \LL}{\partial y'}   \frac{\partial }{\partial u'}\left(\frac{D_x(y) }{D_x(t) }\right) D_x(t)+\LL\frac{\partial D_x(t)}{\partial u'}\nonumber\\
 &= \frac{\partial \LL}{\partial y'}    \frac{D_x(t)\frac{\partial }{\partial u'}(D_x(y))-D_x(y)\frac{\partial }{\partial u'}(D_x(t))}{D_x(t)}  +\LL\frac{\partial D_x(t)}{\partial u'}\nonumber\\
 &=\frac{\partial \LL}{\partial y'} \left(  \frac{\partial }{\partial u'}(D_x(y)) - \frac{D_x(y) }{D_x(t) } \frac{\partial }{\partial u'}(D_x(t)) \right) +\LL\frac{\partial D_x(t)}{\partial u'}\nonumber\\
  &\stackrel{\REF{eqFormulettaChecivoleva}}{=}\frac{\partial \LL}{\partial y'} \left(  \frac{\partial }{\partial u'}(D_x(y)) - y' \frac{\partial }{\partial u'}(D_x(t)) \right) +\LL\frac{\partial D_x(t)}{\partial u'}.\label{refEsercSostFinal0}
\end{align}
It  suffices to observe that
\begin{eqnarray}
\frac{\partial D_x(t)}{\partial u'}&=&\frac{\partial( t_x+u^\prime t_u)}{\partial u'}=t_u,\label{refEsercSostFinal1}\\
\frac{\partial D_x(y)}{\partial u'}&=&\frac{\partial (y_x+u^\prime y_u)}{\partial u'}=y_u.\label{refEsercSostFinal2}
\end{eqnarray}
Substituting \REF{refEsercSostFinal1} and \REF{refEsercSostFinal1} into \REF{refEsercSostFinal0}, one gets
%
%
%
%
%
%
%
%
%
%
%
%
\begin{equation}\label{eqEsercizioQfinale}
 \frac{\partial  {\LL}}{\partial y'}\left(y_u-t_uy'\right)+\LL t_u =0,\nonumber
\end{equation}
which, evaluated at $(0,u(0),u^\prime(0))$, returns \REF{eqEsercFin}, since, by the choice of   $F$, %
\begin{eqnarray*}
t'_\gamma(0)&=&t_u (\gamma(0)),\\
y'_\gamma(0)&=&y_u (\gamma(0)).
\end{eqnarray*}

 \subsection{A $1^\textrm{st}$ order, multi--dimensional example}\label{CondizioniPerSupericie}

Now  we pass to an $n$--dimen\--sional example: as we shall see, the geometric methods used before generalize effortlessly to this case; an analogous generalization of the methods used in Section \ref{secEsempioMotivante}, i.e., defining a metric structure on the space of all functions defined on a compact and connected subset of $\R^n$, would introduce a lot of technical difficulties, obscuring the simple solution of the problem.\par
\begin{theorem}[Generalized transversality conditions]\label{thEgregio}
 Let $Y$ be a closed smooth domain, with nonempty (smooth) boundary, of $\R^{n+1}=(\t,y)$, with\linebreak $\t=(t^1,\ldots, t^n)$, and $\lambda$ be a $1^\mathrm{st}$ order Lagrangian, locally given by\linebreak  $\LL=\LL(\t,y,y_1,\ldots,y_n)$. If $y=y(\t)$ is a critical point for $\S_\lambda$, then,  for any point $\theta\in\partial Y$, the normal vector $\nu_\theta$ to the hypersurface $\partial Y $ must be orthogonal to $H(\t,y,y^1(\t),\ldots,y^n(\t))$, where $H$ is the $\R^{n+1}$--valued function on $G_n^1Y$ given by 
%
%
\begin{equation}
H:=\left( \frac{\partial \LL}{\partial y_1 },\ldots, \frac{\partial \LL}{\partial y_n}, \LL -    y_i \frac{\partial \LL}{\partial y_i }  \right), \nonumber
\end{equation}
and $\theta=(\t,y(\t))$.
\end{theorem}

\begin{proof}
 Choose a change of coordinates    $(\t,y)\stackrel{F}{\longmapsto} (\x,u)$,
\begin{eqnarray*}
\x&=&\x(\t,y),\\
u&=&u(\t,y),
\end{eqnarray*}
such that $F(\partial Y)$ has equation $x^n=0$. In analogy with \REF{eqFormulettaChecivoleva},
\begin{equation}
y_i=\frac{\overline{d}y}{\overline{d}t^i}=\frac{\overline{d}y\wedge \overline{d}_i^{n-1}\t}{\overline{d}t^i\wedge \overline{d}_i^{n-1}\t}=   \frac{\omega_i}{ \overline{d}^{n}\t},
\label{eqSollevamentoPiuBellaN}
\end{equation}
where 
\begin{equation}
\omega_i:=\overline{d}t^1\wedge\cdots\wedge \overline{d}t^{i-1}\wedge \overline{d}y\wedge \overline{d}t^{i+1}\wedge\cdots \overline{d}t^n,\nonumber
\end{equation}
and we use again the convention that horizontal $n$--forms are identified with functions via the (horizontal) volume form $ \overline{d}x^n$  introduced in Remark \ref{remFinale2}. Developing all total differentials appearing in \REF{eqSollevamentoPiuBellaN}, one recovers the familiar formula for the lifting of $F$, as it can be found, e.g. in  \cite{Sym}, Section 1.2.\par
In analogy with \REF{eqExercisio2},
\begin{equation}\label{eqExercisio2N}
\frac{\partial y_i}{\partial u_n}=\frac{\partial }{\partial u_n}\left( \frac{\omega_i}{ \overline{d}^{n}\t}\right),\quad i=1,2,\ldots,n.\nonumber
\end{equation}
Finally,  analogously to \REF{eqLtilde}, we obtain
  $(F^\ast)^{-1}(\lambda) =\widetilde{\LL}\d^n \x$, where\linebreak $\widetilde{\LL}=\widetilde{\LL}(\x,u,u_1,\ldots,u_n)$ is given by
%
%
\begin{equation}
\widetilde{\LL}=(F^\ast)^{-1}(\LL) \overline{d}^{n}\t   \label{eqLtildeN}
\end{equation}
where, as explained by Remark \ref{remFinale2}, $\overline{d}^{n}\t=\frac{ \overline{d}^{n}\t}{ \overline{d}^{n}\x}$ is just an unconventional  way to write down the \virg{total Jacobian} of $F$.\par
Again, $(F^\ast)^{-1}(\lambda)$ determines a variational problem with free boundary values on $F(Y)$ and $\partial F(Y)=F(\partial Y)$, by construction, is the hyperplane $x^n=0$:   by Corollary \ref{corollariosintetizzante},  on such a  hyperplane  the natural boundary conditions  read
\begin{equation}\label{eqEsempio1cambiocoordinateN}
 \frac{\partial \widetilde{\LL}}{\partial u_n} (x^1,\ldots,x^{n-1},0,u(x^1,\ldots,x^{n-1},0),\ldots, u_n(x^1,\ldots,x^{n-1},0))=0.
\end{equation}
Finally, let us write down \REF{eqEsempio1cambiocoordinateN} in terms of the coordinates $(\t,y,y_1,\ldots, y_n)$. 
%
%
%
%
%
%
We compute
\begin{align}
 \frac{\partial \widetilde{\LL}}{\partial u_n}&\stackrel{\REF{eqLtildeN}}{=} \frac{\partial }{\partial u_n}\left(\LL { \overline{d}^{n}\t}\right)\nonumber\\
 &=\frac{\partial \LL}{\partial u_n} { \overline{d}^{n}\t}+\LL\frac{\partial (\overline{d}^{n}\t) }{\partial u_n } \nonumber\\
 &\stackrel{\REF{eqExercisio2}}{=} \frac{\partial \LL}{\partial y_i}   \frac{\partial }{\partial u_n}\left( \frac{\omega_i}{ \overline{d}^{n}\t}\right)  { \overline{d}^{n}\t}  +\LL\frac{\partial (\overline{d}^{n}\t) }{\partial u_n } \nonumber\\
 &= \frac{\partial \LL}{\partial y_i }    \frac{\overline{d}^{n}\t \frac{\partial }{\partial u_n}(\omega_i)-\omega_i\frac{\partial }{\partial u_n }(\overline{d}^{n}\t)}{\overline{d}^{n}\t} +\LL\frac{\partial (\overline{d}^{n}\t) }{\partial u_n } \nonumber\\
 &=\frac{\partial \LL}{\partial y_i } \left(  \frac{\partial }{\partial u_n }(\omega_i) - \frac{\omega_i }{\overline{d}^{n}\t } \frac{\partial }{\partial u_n }(\overline{d}^{n}\t) \right) +\LL\frac{\partial (\overline{d}^{n}\t) }{\partial u_n } \nonumber\\
  &\stackrel{\REF{eqFormulettaChecivoleva}}{=}\frac{\partial \LL}{\partial y_i } \left(  \frac{\partial \omega_i}{\partial u_n } - y_i  \frac{\partial(\overline{d}^{n}\t) }{\partial u_n }  \right) +\LL\frac{\partial (\overline{d}^{n}\t) }{\partial u_n } .\label{refEsercSostFinal0N}
\end{align}
We denote 
\begin{equation}
\nu:=\left( \frac{\partial \omega_1}{\partial u_n },\ldots, \frac{\partial \omega_n}{\partial u_n }, \frac{\partial (\overline{d}^{n}\t) }{\partial u_n }\right)\in\R^{n-1}.\nonumber
\end{equation}

Then, \REF{refEsercSostFinal0N} coincides with $\nu\cdot H$.\par
It remains to show that $\nu$ is indeed the normal vector to $F(\partial Y)$.  To this end, it is convenient to pass to the determinantal  notation for $n$--forms. Namely, 
\begin{align}
 \omega_i=\left\|\begin{array}{ccccc}t_{x_1}^1+u_1 t_u^1 & \cdots & y_{x_1}+u_1y_u & \cdots & t^n_{x_1}+u_1t^n_u \\t_{x_2}^1+u_2 t_u^1 & \cdots & y_{x_2}+u_2y_u & \cdots & t^n_{x_2}+u_2t^n_u \\\vdots & \vdots & \vdots & \vdots & \vdots \\t_{x_{n-1}}^1+u_{n-1} t_u^1 & \cdots & y_{x_{n-1}}+u_{n-1}y_u & \cdots & t^n_{x_{n-1}}+u_{n-1}t^n_u \\t_{x_n}+u_nt_u^1 & \cdots & y_{x_n}+u_ny_u & \cdots & t^n_{x_{n}}+u_{n}t^n_u\end{array}\right\|\nonumber
\end{align}
contains $u_n$ only in the last line: hence, 
\begin{align}
\frac{\partial \omega_i}{\partial u_n }=\left\|\begin{array}{ccccc}t_{x_1}^1+u_1 t_u^1 & \cdots & y_{x_1}+u_1y_u & \cdots & t^n_{x_1}+u_1t^n_u \\t_{x_2}^1+u_2 t_u^1 & \cdots & y_{x_2}+u_2y_u & \cdots & t^n_{x_2}+u_2t^n_u \\\vdots & \vdots & \vdots & \vdots & \vdots \\t_{x_{n-1}}^1+u_{n-1} t_u^1 & \cdots & y_{x_{n-1}}+u_{n-1}y_u & \cdots & t^n_{x_{n-1}}+u_{n-1}t^n_u \\ t_u^1 & \cdots &  y_u & \cdots &  t^n_u\end{array}\right\|.\nonumber
\end{align}
Subtracting from the $j^\textrm{th}$ row   the $n^\textrm{th}$ row multiplied by $u_j$, for all $j=1,\ldots, n-1$, the determinant does not change, i.e., 
\begin{align}\label{refEsercSostFinal1M}
\frac{\partial \omega_i}{\partial u_n }=\left\|\begin{array}{ccccc}t_{x_1}^1  & \cdots & y_{x_1}  & \cdots & t^n_{x_1}  \\t_{x_2}^1  & \cdots & y_{x_2}  & \cdots & t^n_{x_2}  \\\vdots & \vdots & \vdots & \vdots & \vdots \\t_{x_{n-1}}^1  & \cdots & y_{x_{n-1}}  & \cdots & t^n_{x_{n-1}}  \\ t_u^1 & \cdots &  y_u & \cdots &  t^n_u\end{array}\right\|.
\end{align}
Similarly,
\begin{align}\label{refEsercSostFinal2M}
 \frac{\partial (\overline{d}^{n}\t) }{\partial u_n }=\left\|\begin{array}{ccccc}t_{x_1}^1  & \cdots & t^i_{x_1}  & \cdots & t^n_{x_1}  \\t_{x_2}^1  & \cdots & t^i_{x_2}  & \cdots & t^n_{x_2}  \\\vdots & \vdots & \vdots & \vdots & \vdots \\t_{x_{n-1}}^1  & \cdots & t^i_{x_{n-1}}  & \cdots & t^n_{x_{n-1}}  \\ t_u^1 & \cdots &  t^i_u & \cdots &  t^n_u\end{array}\right\|.
\end{align}
Observe that  \REF{refEsercSostFinal1M} and \REF{refEsercSostFinal2M} are the multi--dimensional analogues of  \REF{refEsercSostFinal1} and \REF{refEsercSostFinal2},  respectively. 
In other words, $\nu$ is composed of the $n\times n$ minors (with sign) of the $n\times (n+1)$ matrix
\begin{equation}
(T_1, T_2,\ldots,T_{n-1}, T)^t,\nonumber
\end{equation}
where the $n$ vectors
\begin{eqnarray}
T_i&=& (\t_{x^i},y_{x^i}),\quad i=1,\ldots, n-1,\nonumber\\
T&=&(\t_u,y_u),\nonumber
\end{eqnarray}
form a basis for the tangent space of $F(\partial Y)=\{x_n=0\}$, i.e., $\nu=T_1\times\cdots\times T_{n-1}\times T$.
\end{proof}
Notice that \REF{refEsercSostFinal0N} is formally the same as  \REF{refEsercSostFinal0}: the synthetic language of multi--linear algebra allowed to handle all the \virg{total Jacobian} determinants involved in the proof, without any additional difficulty as compared with   the one--dimensional case.

   \subsection{The   soap film}\label{ProblemaSimpatico}

 We generalize now a classical example that can be found in   Giaquinta and Hildebrandt's book  \cite{Giaq} (Section 2.4). Namely, in the hypotheses of Theorem \ref{thEgregio} above, suppose that $\lambda$ is the (hypersurface) area Lagrangian, i.e., locally,
 \begin{equation}
\LL=\sqrt{1+\sum_{i=1}^n y_i^2}.\nonumber
\end{equation}
Then
\begin{equation}
H=\frac{1}{\sqrt{1+\sum_{i=1}^n y_i^2}}\left( y_1,y_2,\ldots,y_n,-1\right)\nonumber
\end{equation}
is precisely the unit normal vector to the surface $y=y(\t)$.  This proves the last result of this paper.
\begin{corol}\label{corFinale}
 Let $Y\subseteq\R^{n+1}$ be a closed smooth domain with smooth\linebreak nonempty boundary. If a hypersurface $L\in\A_Y$ is a solution of the variational problem with free boundary values determined by  the area functional,   
  then $L$ must intersect orthogonally $\partial Y$ everywhere.
\end{corol}
In particular, Corollary \ref{corFinale} shows that a     soap film, whose boundary is constrained to slide over the inner surface of a fixed domain (e.g., a pipe of arbitrary shape), tends toward  a position of equilibrium where it forms a right angle with the walls of the container (besides, of course, possessing zero mean curvature).

\begin{remark}
If $n=2$ and $Y=D\times\R^2$, where $D\subseteq\R$ is diffeomorphic to a closed disk, then a surface $L$ from Corollary \ref{corFinale} above is forced to be the graph of a constant function $y=y(t_1,t_2)$. Indeed, since $Y$ is a  surface with zero mean curvature, its maximum is attained on $\partial D$. Hence, there exists a point $\theta\in \partial L$, such that $\partial L=L\cap\partial Y$ has negative curvature. But $L$ hits $\partial Y$ orthogonally in $\theta$, thus, along the normal direction to $\partial L$, the surface $L$ must possess positive curvature, i.e.,  there must exist a point $\theta'\in L$, in a neighborhood of $\theta$, such that the $y$--component of $\theta'$ is greater than the $y$--component of $\theta$, thus contradicting the fact that $\theta$ corresponds to a maximum. It follows that $L$ must be the graph of a constant function. It would be nice to generalize this simple observation to multi--dimensional cases.
\end{remark}


\end{document}